\documentclass [11pt] {amsart}
\usepackage{bbm}
\usepackage{amssymb}
\usepackage{epic}
\usepackage{tikz}
\numberwithin{equation}{section}
\newtheorem{theorem}{Theorem}[section]

\newtheorem{lemma}[theorem]{Lemma}
\newtheorem{proposition}[theorem]{Proposition}
\newtheorem{corollary}[theorem]{Corollary}

\theoremstyle{definition}
\newtheorem{definition}[theorem]{Definition}
\newtheorem*{defn}{Definition}

\newtheorem{remark}[theorem]{Remark}

\usepackage{color, colortbl}

\title{$2$-rotundity of some nonseparable\\ abstract interpolation spaces}
\author[Stephen Dilworth]{Stephen Dlworth}
\address{Department of Mathematics, University of South Carolina, Columbia,
SC 29208, USA}
\email{dilworth@math.sc.edu}

\author{Denka Kutzarova}
\address{Department of Mathematics, University of Illinois Urbana-Champaign,
Urbana, IL 61807, USA; Institute of Mathematics and Informatics, Bulgarian Academy of Sciences, Sofia, Bulgaria}
\email{denka@illinois.edu}
\thanks{ The first author was supported by Simons Foundation Collaboration Grant No. 849142.
 The second author was supported by Simons Foundation Collaboration Grant No. 636954.}

\begin{document}
\maketitle

\begin{abstract}
Generalizing a construction due to Argyros and Motakis \cite{AM}, we  define a nonseparable abstract  interpolation space  associated to any given  reflexive space with an unconditional basis together with 
  an increasing  sequence of compact families of finite subsets of an uncountable set.  Under a mild complexity assumption on the families, we prove that the interpolation
space admits a $2$-rotund norm with an uncountable  $1$-unconditional basis.
\end{abstract}

\section{Introduction}\label{sec:Intro}

The  notions of $2$-rotund and weakly $2$-rotund norms   were introduced by Milman \cite{M} and are defined as follows.

 \begin{defn}   Let $X$ be a Banach space. We say that a norm $\|\cdot\|$ on $X$ is $2$-rotund  ($2R$) (resp.\, weakly $2$-rotund ($W2R$))   if for every $(x_n) \subset X$ such that $\|x_n\| \le 1$ ($n \ge 1$) and
$$ \lim_{m,n \rightarrow \infty} \|x_m + x_n\| = 2,$$
there exists $x \in X$ such that $x = \lim_{n \rightarrow \infty} x_n$ strongly (resp.\, weakly).
\end{defn}

It follows from a characterization of reflexivity due to  James \cite{J} that if  $X$ admits an equivalent $W2R$ norm then $X$ is reflexive. H\'ajek and Johanis proved the converse:   every reflexive Banach space admits an equivalent $W2R$ norm \cite{HJ}. Odell and Schlumprecht \cite{OS} proved that every separable   reflexive Banach space $X$ admits an equivalent $2R$ norm (cf.\,\cite{G}). However,  it is an open question whether every  reflexive Banach space admits an equivalent $2R$ norm.

In \cite{DK}, motivated by Schreier hierarchy introduced by Alspach and Argyros in \cite{AA} and a family of sets introduced by Benyamini and Starbird in \cite{BS}, we presented a general method for defining, for each countable ordinal $\alpha$, a family $\mathcal{F}_\alpha$ for certain uncountable sets $\Gamma$. The construction is similar to that of the transfinitely defined families introduced by Argyros and Motakis in \cite{AM}. We proved that the generalized Baernstein space $B(\mathcal{F}_\alpha)$ has an equivalent $2R$ norm. (Baernstein defined his space in \cite{B}.) As an application, we obtained that the nonseparable reflexive space in \cite{KT} has an equivalent $2R$ norm. We do not know if the same is true for $B(\mathcal{F})$ for arbitrary compact hereditary family $\mathcal{F}$ on $\Gamma$.

Inspired by a paper of Figiel and Johnson \cite{FJ}, we considered in \cite{DKM}  reflexive spaces with a (countable) unconditional (resp.\ symmetric) basis and asked if they admit a $2R$ 1-unconditional (resp.\ 1-symmertic) norm. Modifying the proof of Odell and Schlumprecht we obtained a positive result for 1-unconditionality, which we shall use in the present paper.

Tsirelson space \cite{T} is one of the most important spaces in the theory of Banach spaces. L\'{o}pez-Abad and Todorcevic \cite{LAT} showed that in the nonseparable setting there is no Tsirelson type space such that the norm satisfies an implicit formula in the spirit of \cite{FJ}. In \cite{AM}, Argyros and Motakis provided examples of nonseparable spaces that have some of the essential properties of Tsirelson space, e.g. being reflexive and admitting only $\ell_1$ as a spreading model. They used the construction inspired by the real interpolation method that appears in \cite{DFJP}, also in \cite{FJ}.

In the present paper we show that the nonseparable reflexive spaces in the spirit of Argyros and Motakis admit an equivalent $2R$ norm. The main result is proved in Section \ref{sec:2Rrenorm}. In Section \ref{sec:Refl} we prove reflexivity of the spaces in question under a certain assumption.  For the main result we apply an inductive approach in constructing the $2R$ norm, which takes into account the complexity of the corresponding underlying families of sets.

\section{Preliminary results}\label{sec:PrRes}
We shall use the following  characterization of $2$-rotundity (see e.g., \cite[II.6.4]{DGZ} or \cite{HJ}): $\|\cdot\|$ is a $2R$ norm on $X$  if for all  $(x_n)_{n=1}^\infty \subset X$ such that \begin{equation} \label{eq: alternativedef}
\lim_{m,n \rightarrow \infty} [\| x_m + x_n\|^2 - 2(\|x_m\|^2 + \|x_n\|^2)] = 0, \end{equation}
 there exists $x \in X$ such that $x=\lim_{n \rightarrow \infty} x_n$ strongly.

Day \cite{D}  introduced the  norm $\|\cdot \|_{\textrm{Day}}$ on $c_0$  defined by
$$\|(a_n)_{n=1}^\infty\|_{\textrm{Day}} = (\sum_{n=1}^\infty 4^{-n} a_n^{*2})^{1/2},
$$
where $(a_n^*)_{n=1}^\infty$ is the non-increasing rearrangement of $(|a_n|)_{n=1}^\infty$.
Let $(Y,\|\cdot\|)$ be a reflexive Banach space with normalized basis $(e_n)_{n=1}^\infty$. We define an equivalent norm
on   $Y$ thus:
\begin{equation}\label{eq: w2Rrenorming}   |\!|\!|\sum_{n=1}^\infty a_n e_n  \  |\!|\!| =( \|\sum_{n=1}^\infty a_ne_n\|^2 +  \|(a_n)_{n=1}^\infty\|_{\textrm{Day}}^2)^{1/2} .
\end{equation}

 We will use the  following result of  H\'ajek and Johanis.  It is a consequence of Theorem~3 and Corollary~4 of \cite{HJ} and the reflexivity of $Y$. (Here $\|\sum_{n=1}^\infty a_n e_n\|_\infty  = \sup_{n \ge 1} |a_n|$ as usual.)

\begin{theorem}  \label{lem: HJresult}
Suppose $(y_n)_{n=1}^\infty \subset Y$ satisfies
\begin{equation} \label{eq: triplenormcondition_2}
 \lim_{m,n \rightarrow \infty}[ |\!|\!|y_n + y_m |\!|\!|^2- 2(|\!|\!|y_n|\!|\!|^2 + |\!|\!|y_m|\!|\!|^2) ]= 0.
 \end{equation}
Then there exists $y \in Y$ such that
\begin{equation*}
 y_n\rightarrow y\quad \text{weakly as $n \rightarrow \infty$}
\end{equation*} and
$$
\lim_{n \rightarrow \infty} \|y_n - y\|_\infty = 0.
$$
\end{theorem}

For $K \ge 1$,  a basis $(e_n)_{n=1}^\infty$  is  $K$-unconditional if
$$
\|\sum_{n=1}^\infty \pm  a_n e_n\| \le K\|\sum_{n=1}^\infty a_n e_n\|
$$
for all scalars $(a_n)_{n=1}^\infty$ and all choices of signs. The basis is $K$-symmetric if
$$
\|\sum_{n=1}^\infty \pm  a_{\sigma(n)} e_n\| \le K\|\sum_{n=1}^\infty a_n e_n\|
$$
for all scalars $(a_n)_{n=1}^\infty$, all choices of signs, and all permutations $\sigma\colon \mathbb{N} \rightarrow \mathbb{N}$.

We refer the reader to \cite{LT1} for other  unexplained  Banach space  notation and terminology.

\begin{theorem}[\cite{DKM}]\label{t22}
Let $Y$ be a reflexive Banach space with an unconditional basis $(e_n)_{n=1}^\infty$. Then $Y$ admits an equivalent $1$-unconditional $2R$ norm.
\end{theorem}

Let $\Gamma$ be an infinite set. Throughout,  $\mathcal{F}$ denotes a collection of finite subsets of $\Gamma$ satisfying the following: \begin{itemize}
\item $\mathcal{F}$ contains all singletons;
\item $\mathcal{F}$ is hereditary, i.e., if $F \in \mathcal{F}$ and $G \subseteq F$ then $G \in \mathcal{F}$;
\item $\mathcal{F}$ is compact, i.e., $\{ 1_F \colon F \in \mathcal{F}\}$
is a compact subset of $\{0,1\}^{\Gamma}$ in the topology of pointwise convergence.
\end{itemize}

Now we shall define more specific families of sets satisfying the above three conditions.
\begin{itemize} \item
 Let  $S$ be any set of cardinality at least  $2$ and let $\overline{S} := S^{\mathbb{N}}$.  \label{sec: transfinite}  \item
For distinct $p = (p(i))_{i=1}^\infty \in \overline{S}$ and $q = (q(i))_{i=1}^\infty \in \overline{S}$, let
$d(p,q) = 1$ if $p(1) \ne q(1)$ and, for $k \ge 2 $, let $d(p,q) =  k$ if $p(k) \ne q(k)$ and $p(j) = q(j)$ for $1 \le j \le k-1.$
\item For $A \subset \overline{S}$, with $|A| \ge 2$, let $$A^\sharp = \min \{d(p,q) \colon p,q \in A, p \ne q\}.$$
We define, for each countable ordinal $\alpha$, a  hereditary family $\mathcal{F}_\alpha$ of finite subsets of  $\overline{S}$.
\item Let $$\mathcal{F}_0  = \{ \emptyset \} \cup \{\{p\} \colon p \in \overline{S}\}.$$
\item  If $k \ge 1$ and  $\mathcal{F}$ is any collection of finite subsets of $\overline{S}$ satisfying the conditions set out in the Introduction, let
$$\mathcal{F}^{(k)} = \mathcal{F}_0 \cup \{A \in \mathcal{F} \colon A^{\sharp} \ge k\}.$$
Note that since  $\mathcal{F}$ is hereditary,  $\mathcal{F}^{(k)}$ is also hereditary.
\item If $\alpha = \beta^+$ is a successor ordinal, let $\mathcal{F}_\alpha$ be any hereditary family satisfying the following:
\begin{itemize} \item $\mathcal{F}_\beta \subseteq \mathcal{F}_\alpha$.
\item If  $A \in \mathcal{F}_\alpha$ and  $|A| \ge 2$, then  there exist  $A_i \in \mathcal{F}_\beta$ ($1 \le i \le A^{\sharp}$) such that
$$ A = \cup_{i=1}^{A^{\sharp}} A_i.$$
\end{itemize}
\item If $\alpha$ is a limit ordinal, choose $\alpha_r \uparrow \alpha$ ($r \ge 1$) and define
$$ \mathcal{F}_\alpha = \cup_{r=1}^\infty \mathcal{F}^{(r)}_{\alpha_r}.$$ Note that, for each $k \ge 1$,
$$\mathcal{F}^{(k)}_\alpha = \cup_{r=1}^\infty \mathcal{F}^{(r \vee k)}_{\alpha_r},$$
where $r \vee k := \max(r,k)$.
\end{itemize}
The above construction includes the compact families from \cite{AM}. Next we show that this more general construction also produces compact families.

\begin{proposition}\label{t51}
The hereditary families $\mathcal{F}_\alpha$ are compact in the topology of pointwise convergence for all $\alpha$.
\end{proposition}

\begin{proof} 
Let $(F_u)$ be a net in $\mathcal{F}_\alpha$. By passing to a subnet we may assume that either $F_u\to\emptyset$ or $F_u\to\{p\}$ or $\{p,q\}\subseteq F_u$ for all $u$, where $p,q\in\overline S$ and $p\ne q$. In particular, if $\alpha=0$ then either $F_u\to\emptyset$ or $F_u\to\{p\}$, so $\mathcal{F}_0$ is compact.

Now suppose the result holds for all $\gamma<\alpha$. 

First, suppose $\alpha=\beta^+$ is a successor ordinal. If $\{p,q\}\subseteq F_u$ for all $u$, then $F_u=\mathop{\cup}\limits_{i=1}^{d(p,q)}F_{u,i}$ for some $F_{u,i}\in \mathcal{F}_\beta$. We may assume that $(F_{u,i})_{i=1}^{d(p,q)}$ are disjoint for each $u$. By compactness of $\mathcal{F}_\beta$, passing to a subnet we may assume $F_{u,i}\to F_i\in\mathcal{F}_\beta$, $1\le i\le d(p,q)$. The limit sets will be eventually subsets of the net sets, so $F_i\in\mathcal{F}_\beta$, since $\mathcal{F}_\beta$ is hereditary. Thus, $\mathop{\cup}\limits_{i=1}^{d(p,q)}F_{i}$ will be eventually a subset of $F_u$, hence it is in $\mathcal{F}_\alpha$, since $\mathcal{F}_\alpha$ is hereditary.

Now suppose $\alpha$ is a limit ordinal. If $\{p,q\}\subseteq F_u$ for each $u$ then $F_u^\sharp \le d(p,q)$ and hence $F_u \in \mathop{\cup}\limits_{r=1}^{d(p,q)}\mathcal{F}_{\alpha_r}^{(r)}$, which is compact by  inductive hypothesis. So, passing to a subnet, $F_u\to F\in\mathop{\cup}\limits_{r=1}^{d(p,q)}\mathcal{F}_{\alpha_r}^{(r)}\subseteq\mathcal{F}_\alpha$. Thus $\mathcal{F}_\alpha$ is compact.
\end{proof}

For our main result we shall use only the particular case of $\mathcal{F}_j$, for integers $j\ge 0$.

Let $X_j$, $j\ge 0$, be the generalized Schreier space \cite{S} based on $\mathcal{F}_j$ with norm $\|\cdot\|_j$ for $x\in c_{00}(\overline{S})$, where $c_{00}(\overline{S})$ is the linear space of vectors with finitely many non-zero coordinates in $\overline{S}$,
$$
\|x\|_j=\sup_{F\in \mathcal{F}_j} \sum_{\gamma\in F} |x_\gamma|.
$$
$X_j$ is the closure of $c_{00}(\overline{S})$ in $\|\cdot\|_j$.

For a finite $F\subseteq \overline{S}$, and $x=(x_\gamma)\in c_{00}(\overline{S})$, denote
$$
\Theta_F(x)=\sum_{\gamma\in F}|x_\gamma|.
$$

Define an equivalent norm $|\!|\!|\cdot|\!|\!|_j$ on $X_j$ by
$$
|\!|\!|x|\!|\!|_j=\sup\|(\Theta_{F_i}(x))\|_{\textrm{Day}},
$$
where the supremum is taken over all disjoint $F_1,F_2,\dots,F_N$, $N\in\mathbb{N}$, in $\mathcal{F}_j$.

\begin{lemma}\label{l23}  Let $j \ge 1$.
Suppose $x\in X_j$, $x_n\in X_j$ ($n\ge 1$), $\|x_n\|_j>\delta>0$ for all $n\ge 1$, and $\|x_n\|_\infty\to 0$ as $n \rightarrow \infty$.. Then
$$
\lim_{n\to \infty}|\!|\!|x+x_n|\!|\!|_j>|\!|\!|x|\!|\!|_j.
$$
\end{lemma}

\begin{proof}
Choose a finite set $E\subseteq \overline{S}$ such that
$$
\|x-x\cdot 1_E\|_j<\frac\delta2.$$

 Let $N:=|E|$.  Choose disjoint sets $F_1,\dots,F_m$ in $\mathcal{F}_j$ such that
$$
\|(\Theta_{F_i}(x))_{i=1}^m\|_{\textrm{Day}}^2>|\!|\!|x|\!|\!|_j^2-\frac{\delta^2}4 4^{-N-1}$$
Let $$y_n = x_n - x_n \cdot 1_{\cup_{i=1}^m F_i}.$$ Since $\|x_n\|_\infty \rightarrow 0$, it follows that
$\|x_n - y_n\|_j \rightarrow 0$. 
 Note that if $F \cap E = \emptyset$ then $\Theta_F(x)< \delta/2$.

Thus, since the $F_i$'s are disjoint, at most $N$ of them   can satisfy
$$\Theta_{F_i}(x) \ge \frac{\delta}{2}.$$
Moreover, since  $\|x_n\|_j>\delta$  it follows that for all sufficiently large $n$ there exists $F^n_{m+1}\in\mathcal{F}_j$ disjoint from $\mathop{\cup}\limits_{i=1}^m F_i$ such that $\Theta_{F^n_{m+1}}(y_n)>\delta$. Moreover, we may assume that
$$\lim_{n \rightarrow \infty}  \Theta_{F^n_{m+1}}(x) = 0$$ 
since $\|x\|_j < \infty$.   
Let $\varepsilon>0$. Then for all sufficiently large $n$,
\begin{align*} 
|\!|\!|x+y_n|\!|\!|_j^2\ge & \|(\Theta_{F_1}(x+y_n),\dots, \Theta_{F_m}(x+y_n),\Theta_{F^n_{m+1}}(x+y_n))\|_{\textrm{Day}}^2\\
&= \|(\Theta_{F_1}(x),\dots, \Theta_{F_m}(x),\Theta_{F^n_{m+1}}(x+y_n))\|_{\textrm{Day}}^2\\
&\ge  \|(\Theta_{F_1}(x),\dots, \Theta_{F_m}(x),\Theta_{F^n_{m+1}}(y_n))\|_{\textrm{Day}}^2 - \varepsilon\\
&\ge \|(\Theta_{F_i}(x))_{i=1}^m\|_{\textrm{Day}}^2 + \left(\delta^2-\frac{\delta^2}4\right)4^{-N-1}-\varepsilon.
\end{align*}The last line uses the fact that the $(N+1)^{\rm st}$ largest $\Theta_{F_i}(x)<\frac\delta2$, so replacing it by $\Theta_{F^n_{m+1}}(y_n)>\delta$ increases the square of the Day  norm by  at least $\left(\delta^2-\frac{\delta^2}4\right)4^{-N-1}$.

 Hence, for all sufficiently large $n$, \begin{align*}
|\!|\!|x+y_n|\!|\!|_j^2\ge &  \|(\Theta_{F_i}(x))_{i=1}^m\|_{\textrm{Day}}^2 + \left(\delta^2-\frac{\delta^2}4\right)4^{-N-1}-\varepsilon\\
&\ge |\!|\!|x|\!|\!|_j^2-\frac{\delta^2}4 4^{-N-1} + \left(\delta^2-\frac{\delta^2}4\right)4^{-N-1}-\varepsilon\\
&=  |\!|\!|x|\!|\!|_j^2+\frac{\delta^2}2 4^{-N-1} -\varepsilon. 
\end{align*} Finally, since $\|x_n - y_n\|_j \rightarrow 0$ and $\varepsilon>0$ is arbitrary, the required result follows.

\end{proof}

\begin{lemma}\label{l24}
Let $j\ge 1$. Suppose $z\in c_{00}(\overline{S})$, $x_n\in c_{00}(\overline{S})$, and $\|x_n\|_{j-1}\to 0$ as $n\to \infty$. Suppose $F\in\mathcal{F}_j$ satisfies $|F\cap\operatorname{supp}(z)|\ge 2$. Then $\lim\limits_{n\to\infty}\Theta_F(x_n)=0$ uniformly over all such $F$.
\end{lemma}

\begin{proof}
Denote $N=\max\{d(p,q)\colon p,q\in\operatorname{supp}(z), p\ne q\}$. Then $F$ is of the form $F=\mathop{\cup}\limits_{i=1}^N F_i$, where $F_i\in\mathcal{F}_{j-1}$. So
$$
\Theta_F(x_n)\le\sum_{i=1}^N \Theta_{F_i}(x_n)\le N\|x_n\|_{j-1}\to 0
$$
uniformly over $F$ as $n\to\infty$.
\end{proof}

\begin{lemma}\label{l25} Let $j \ge 1$.
Suppose $x\in X_j$ and $x_n\in c_{00}(\overline{S})$ ($n\ge 1$), with $\operatorname{supp}(x_n)\cap\operatorname{supp}(x_m)=\emptyset$ if $m\ne n$, satisfy $\|x_n\|_{j-1}\to 0$ as $n \rightarrow \infty$, $\lim\limits_{n\to\infty}|\!|\!|x+x_n|\!|\!|_j\le 1$, and $\lim\limits_{m,n\to\infty}|\!|\!|2x+x_m+x_n|\!|\!|_j=2$. Then $\|x_n\|_j \to 0$ as $n\to\infty$.
\end{lemma}

\begin{proof} 
Let $\varepsilon>0$ and $\delta=\delta(\varepsilon):= c\varepsilon^2$, where $c>0$ is sufficiently small. Fix $n\ge 1$ and fix $m>n$. By the definition of the norm $|\!|\!|\cdot|\!|\!|_j$, there exist $N\ge 1$ and disjoint sets $F_1,\dots,F_N$ in $\mathcal{F}_j$ such that
$$
|\!|\!| 2x+x_n+x_m|\!|\!|_j\le \|(\Theta_{F_i}(2x+x_n+x_m))_{i=1}^N\|_{\ell_2(w)}+\delta,
$$
where $\ell_2(w)$ is the weighted $\ell_2$ norm
$$
\|(a_i)_{i=1}^N\|_{\ell_2(w)}=\Big[\sum_{i=1}^N a_i^2w_i\Big]^{1/2}
$$
and $(w_1,\dots,w_N)$ is some rearrangement of $(4^{-i})_{j=1}^N$.
So for $m$ and $n$ sufficiently large, we have
$$
\|(\Theta_{F_i}(2x+x_n+x_m))_{i=1}^N\|_{\ell_2(w)}>2-2\delta,
$$
and $|\!|\!|x+x_n|\!|\!|_j\le 1+\delta$, $|\!|\!|x+x_m|\!|\!|_j\le 1+\delta$. Thus, we also have
$$
\|(\Theta_{F_i}(x+x_n))_{i=1}^N\|_{\ell_2(w)}\le|\!|\!|x+x_n|\!|\!|_j\le 1+\delta
$$
and
$$
\|(\Theta_{F_i}(x+x_m))_{i=1}^N\|_{\ell_2(w)}\le|\!|\!|x+x_m|\!|\!|_j\le 1+\delta.
$$ Since  $\delta = c\varepsilon^2$ and the modulus of convexity of $\ell_2(w)$ is of power type $2$, it follows that provided $c$ is sufficiently small,
$$
\|(\Theta_{F_i}(x_n-x_m))_{i=1}^N\|_{\ell_2(w)}<\frac\varepsilon4.
$$
Let $A=\{1\le i\le N \colon |F_i\cap\operatorname{supp}(x_n)|\le 1\}$. Then $\Theta_{F_i}(x_n)\le\|x_n\|_{\infty}$ for all $i\in A$ and so
$$
\|(\Theta_{F_i}(x_n))_{i\in A}\|_{\ell_2(w)}\le\|x_n\|_{\infty}.
$$
Hence, provided $n$ is sufficiently large, $\|(\Theta_{F_i}(x_n))_{i\in A}\|_{\ell_2(w)}<\delta$, since $\|x_n\|_\infty\le\|x_n\|_{j-1}\to 0$ as $n\to\infty$.

Let $B=A^c=\{1\le i\le N\colon |F_i\cap\operatorname{supp}(x_n)|\ge 2\}$. By Lemma \ref{l24}, provided $m>n$ is sufficiently large (i.e., $m>M(n)$), $\Theta_{F_i}(x_m)<\delta$ for all $i\in B$ (uniformly over all possible choice of $F_i$). Thus, $\|(\Theta_{F_i}(x_m))_{i\in B}\|_{\ell_2(w)}<\delta$. Hence,
\begin{align*}
\|(\Theta_{F_i}(x_n))_{i=1}^N\|_{\ell_2(w)}&\le
\|(\Theta_{F_i}(x_n))_{i\in A}\|_{\ell_2(w)}+
\|(\Theta_{F_i}(x_n))_{i\in B}\|_{\ell_2(w)}\\
&\le \delta+\|(\Theta_{F_i}(x_n-x_m))_{i\in B}\|_{\ell_2(w)}+
\|(\Theta_{F_i}(x_m))_{i\in B}\|_{\ell_2(w)}\\
&\le\delta+\frac\varepsilon4+\delta=2\delta+\frac\varepsilon4.
\end{align*}
Thus,
\begin{align*}
\|(\Theta_{F_i}(x_m))_{i=1}^N\|_{\ell_2(w)}&
\le \|(\Theta_{F_i}(x_n))_{i=1}^N\|_{\ell_2(w)}+
\|(\Theta_{F_i}(x_n-x_m))_{i=1}^N\|_{\ell_2(w)}\\
&<\left(2\delta+\frac\varepsilon4\right)+\frac\varepsilon4=2\delta+\frac\varepsilon2.
\end{align*}
Recall that $\|(\Theta_{F_i}(2x+x_n+x_m))_{i=1}^N\|_{\ell_2(w)}>2-2\delta$.
Hence by the triangle inequality,
$$
\|(\Theta_{F_i}(2x))_{i=1}^N\|_{\ell_2(w)}>2-2\delta-\left(2\delta+\frac\varepsilon2\right)-
\left(2\delta+\frac\varepsilon4\right)
=2-6\delta-\frac34\varepsilon>2-\varepsilon
$$
if $c$ is sufficiently small. Therefore,
$$
|\!|\!|x|\!|\!|_j\ge \|(\Theta_{F_i}(x))_{i=1}^N\|_{\ell_2(w)}>1-\frac\varepsilon2.
$$
Since $\varepsilon>0$ is arbitrary, $|\!|\!|x|\!|\!|_j\ge 1$. But $\lim\limits_{n\to\infty}|\!|\!|x+x_n|\!|\!|_j\le 1$ and $\|x_n\|_\infty\le\|x_n\|_{j-1}\to 0$ by assumption. So by Lemma \ref{l23}, we obtain by contradiction that $\|x_n\|_j\to 0$ as $n\to \infty$.
\end{proof}

\section{$2R$ renorming}\label{sec:2Rrenorm}
Let $(e_j)$ be an unconditional basis for a reflexive Banach space $Y$. Let $X$ be the complection of $c_{00}(\overline{S})$ equipped with the norm

$$
\|x\|=\Big\|\sum_{j=0}^\infty\frac{\|x\|_j}{2^j}e_j\Big\|_{Y}. 
$$  We may regard $X$ as a kind of abstract interpolation space (cf. \cite{DFJP} and the mongraph \cite{Beau}). 
The space constructed by Argyros and Motakis \cite{AM} is a particular case of such spaces $(X,\|\cdot\|)$. 

By renorming $Y$ we may assume according to Theorem \ref{t22} that $\|\cdot\|_Y$ is $1$-unconditional and has the $2R$ property.

\begin{remark} $(X,\|\cdot\|)$ does not have the $2R$ property. To see this, we may assume $\{0,1\} \subseteq S$.  For $n \ge 0$, define $(\gamma_n) \in \overline{S}$ as follows:
$$ \gamma_0 = (0,0,\dots), \gamma_1 = (1,0,0,\dots),  \gamma_2 = (1,1,0,0,\dots), \gamma_3 = (1,1,1,0,0,\dots),$$
etc.   Since $d(\gamma_0, \gamma_n) = 1$ for $n \ge 1$, it follows that $\{\gamma_0, \gamma_n\} \notin \mathcal{F}_j$ for all $n \ge 1$ and  $j \ge 0$.  Hence for $j \ge 0$,
$$\|1_{\{\gamma_0\}}\|_j = \|1_{\{\gamma_0\}}+ 1_{\{\gamma_n\}}\|_j =1. $$
Let $x_n = 1_{\{\gamma_0\}}+ 1_{\{\gamma_n\}}$. For $j \ge 0$,
$$ \|x_n\|_j =1,  \|x_n+x_m\|_j =2, \|x_n - x_m\|_j\ge 1 \qquad (m \ne n).$$ Hence
$$\|x_n+x_m\|_X = \|x_n\|_X + \|x_m\|_X.$$
So $X$ is not rotund and hence not $2R$. Moreover,  if $X$ is reflexive then $x_n \rightarrow 1_{\{\gamma_0\}}$ weakly but not strongly and
$\|x_n\|_X = \|1_{\{\gamma_0\}}\|_X$ ($n\ge1$).
So  $X$ fails the Kadec-Klee property.
\end{remark}

We equip $(X,\|\cdot\|)$ with the equivalent norm: 
$$
\|x\|_*=(\|x\|^2+\|x\|^2_{\textrm{Day}}+\sum_{j=0}^\infty \varepsilon_j
|\!|\!|x|\!|\!|_j^2)^{1/2},
$$
where $\varepsilon_j\to0$ fast enough to ensure convergence, e.g., $\varepsilon_j = 2^{-2j}$.

\begin{theorem}
The following are equivalent.

$(1)$ $(X,\|\cdot\|)$ is reflexive.

$(2)$ $(X, \| \cdot\|_*)$ has $2R$ property.
\end{theorem}

(In the next section we provide sufficient conditions for reflexivity.)

\begin{proof}$(1) \Rightarrow (2)$.
Suppose
$$
\lim_{m,n\to\infty}[\|x_m+x_n\|^2_*-2(\|x_m\|_*^2+\|x_n\|_*^2)]=0.
\leqno(*)
$$
Then by Theorem \ref{lem: HJresult} \cite{HJ}, $x_n\to x$ weakly for some $x\in X$ and $\|x_n-x\|_\infty\to 0$. In particular, this implies that
$$
\lim_{m,n\to\infty} [\|x_m+x_n\|^2-2(\|x_m\|^2+\|x_n\|^2)]=0.
$$
Since the basis of $Y$ is $1$-unconditional, it follows that $y_n:=\sum_{j=0}^\infty\frac{\|x_n\|_j}{2^j}e_j \in(Y,\|\cdot\|_Y)$ and
\begin{align*}
\|y_m+y_n\|_Y&=\Big\|\sum_{j=0}^\infty\frac1{2^j}(\|x_m\|_j+\|x_n\|_j)e_j\Big\|_Y\\
&\ge\Big\|\sum_{j=0}^\infty\frac1{2^j}\|x_m+x_n\|_je_j\Big\|_Y=\|x_m+x_n\|.
\end{align*}
Thus,
$$
\lim_{m,n\to\infty}[\|y_m+y_n\|^2_Y-2(\|y_m\|^2_Y+\|y_n\|_Y^2)]=0.
$$
It follows from the $2R$-property of $(Y,\|\cdot\|_Y)$ that there is $y\in Y$ such that $\|y_n-y\|_Y\to 0$ as $n\to\infty$. In particular the ``tail''
$$
\Big\|\sum_{j=N+1}^\infty \frac{\|x_n\|_j}{2^j}e_j\Big\|_Y\to 0
$$
uniformly over $n$ as $N\to\infty$. 

Let $\tilde x_n:=x_n-x$, so that $x_n=x+\tilde x_n$. We wish to show that $\|\tilde x_n\|_*\to 0$ as $n\to\infty$. Note that
$$
\Big\|\sum_{j=N+1}^\infty\frac{\|\tilde x_n\|_j}{2^j}e_j\Big\|_Y\le
\Big\|\sum_{j=N+1}^\infty\frac{\|x_n\|_j}{2^j}e_j\Big\|_Y+
\Big\|\sum_{j=N+1}^\infty\frac{\|x\|_j}{2^j}e_j\Big\|_Y\to 0
$$
uniformly over $n$ as $N\to\infty$. So it suffices to show that, for each $j\ge0$, $\|\tilde{x}_n\|_j\to 0$ as $n\to\infty$. (Here $\|x\|_0=\|x\|_\infty)$, which
we prove by induction.

The statement is true for $j=0$ since $\|\tilde x_n\|_\infty\to 0$ by Theorem \ref{lem: HJresult} \cite{HJ}. So suppose the result is true for $j-1\ge 0$. Now condition $(*)$ implies that for all $j\ge 0$ 
$$
\lim_{m,n\to\infty}[|\!|\!|x_m+x_n|\!|\!|_j^2
-2(|\!|\!|x_m|\!|\!|_j^2
+|\!|\!|x_n|\!|\!|_j^2]=0,
$$
so
$$\lim_{m,n\to\infty}[|\!|\!|2x+\tilde x_m+\tilde x_n|\!|\!|_j^2
-2(|\!|\!|x +\tilde x_m|\!|\!|_j^2
+|\!|\!|x+\tilde x_n|\!|\!|_j^2]=0.
$$
Suppose that $(\|\tilde x_n\|_j)_{n=1}^\infty$ does not tend to $0$. By passing to a subsequence and approximating by finitely supported vectors, we may assume $\tilde x_n\in c_{00}(\overline S)$, $\operatorname{supp}(\tilde x_n)\cap \operatorname{supp}(\tilde x_m)=\emptyset$ for $m\ne n$, $\inf_{n\ge 1}\|\tilde x_n\|_j=\delta>0$ and that
$$
\lim_{m,n\to\infty}|\!|\!|2x+\tilde x_m+\tilde x_n|\!|\!|_j =2\lim_{n\to\infty}|\!|\!|x+\tilde x_n|\!|\!|_j,
$$
where both limits exist.

By the inductive hypothesis $\|\tilde x_n\|_{j-1}\to 0$ as $n\to\infty$. So, by Lemma \ref{l25}, $\|\tilde x_n\|_j\to 0$, which contradicts the fact that $\inf_{n}\|\tilde x_n\|_j>0$. So the result holds for $j$.

$(2) \Rightarrow (1)$ is immediate since $\|\cdot\|$ and $\|\cdot\|_{*}$ are equivalent and the $2R$ property implies reflexivity.
\end{proof}

\section{Reflexivity}\label{sec:Refl}

In this section we impose restrictions on the families $\mathcal{F}_n$, $n\in\mathbb{N}$.   The complexity of $(\mathcal{F}_n)_{n=1}^\infty$ must increase sufficiently with $n$  to ensure  reflexivity of $X$.

\begin{definition}
Let $S$ be a finite set. We say that $(\mathcal{F}_n)_{n=1}^\infty$ satisfy the \textit{inclusion condition} if for each  $n\ge1$ and $M\in\mathbb{N}$, whenever $(A_i)_{i=1}^\infty$ is a sequence of disjoint sets such that $A_i\in \mathcal{F}_n$ and $A_i^{\#} \to \infty$ as $i \to \infty$, then there exist $i_1<i_2<\cdots<i_M$ such that $\mathop{\cup}\limits_{k=1}^M A_{i_k}\in \mathcal{F}_{n+1}$. For $n=0$, the inclusion condition requires that  for every sequence $(\gamma_i)\subseteq \overline S$ and for each $M\in\mathbb{N}$ there exist $i_1<i_2<\cdots<i_M$ such that $\{\gamma_{i_k}\colon 1\le k\le M\}\in\mathcal{F}_1$. (The inclusion condition makes sense when $S$ is finite.)
\end{definition}

\begin{remark}
Note that the maximal families $(\mathcal{F}_n)$ consisting of all finite subsets $A\subseteq\overline S$ such that $A=\mathop{\cup}\limits_{i=1}^{A^{\#}} B_i$, where $B_i\in\mathcal{F}_{n-1}$, $1\le i\le A^{\#}$, satisfy the inclusion condition when $S$ is finite.
\end{remark}

Let again $Y$ be a reflexive space with a $1$-unconditional basis $(e_j)_{j=0}^\infty$ and consider the corresponding space $X$.

Define $P_n : X\to \sum_{j=0}^n \oplus X_j$, where $X_j$ is the completion of $(c_{00}(\overline S),\|\cdot\|_j)$, so
$$
\|P_n(x)\|=\Big\|\sum_{j=0}^n\frac{\|x\|_j}{2^j}e_j\Big\|_Y.
$$
Note that
$$
\frac{\|x\|_n}{2^n}\le \|P_n(x)\|\le 2\|x\|_n.
$$

\begin{lemma}\label{l41}
Suppose that $S$ is finite and that the sequence $(\mathcal{F}_n)$ of families of finite subsets of $\overline S$ satisfies the inclusion condition. Then, for each $n\ge 0$, $P_n$ is strictly singular.
\end{lemma}

\begin{proof}
It is known that the generalized Schreier spaces $(X_j, \|\cdot\|_j)$ are $c_0$-saturated. So it suffices to check that $P_n$ does not fix any copy of $c_0$ in $X$. The proof is by induction on $n$. Let $n=0$. Note that $X_0=c_0(\overline S)$. Suppose that $(y_i)_{i=1}^\infty\subseteq X$ is equivalent to the unit vector basis (u.v.b.) of $c_0$. By approximation and passing to a subsequence we may assume that the supports of $y_i$'s are finite and disjoint. Suppose that $P_0$ is an isomorphism on $[(y_i)_{i=1}^\infty]$. Hence there exists $c>0$ such that $\|y_i\|_0\ge c$, $i\in\mathbb{N}$. Since $X_0=c_0(\overline S)$ it follows that there exists $\gamma_i\in\operatorname{supp}(y_i)$ such that $|e^*_{\gamma_i}(y_i)|\ge c$. Then by the inclusion condition for $n=0$, for every $M\in\mathbb{N}$ there is a subset $A_M\subseteq\{\gamma_i\colon i\in\mathbb{N}\}$, $|A_M|=M$, $A_M\in\mathcal{F}_1$. Thus,
$$
\Big\|\sum_{\gamma_i\in A_M} y_i\Big\|_1\ge
\sum_{\gamma_i\in A_M}|e^*_{\gamma_i}(y_i)|\ge|A_M|c=Mc
$$
and hence,
$$
\Big\|\sum_{\gamma_i\in A_M} y_i\Big\|_X\ge\frac{Mc}2,
$$
which contradicts the fact that $(y_i)_{i=1}^\infty\subseteq X$ is equivalent to the u.v.b. of $c_0$, since $M$ can be arbitrarily large.

We prove the case $n\ge 1$ by induction. Suppose the result holds for $0\le m\le n-1$. As above, assume $(y_i)_{i=1}^\infty\subseteq X$ is equivalent to the u.v.b. of $c_0$ and that $\left(\operatorname{supp}(y_i)\right)_{i=1}^\infty$ are finite and disjoint. Suppose that $P_n$ is an isomophism on $\left[(y_i)\right]_{i=1}^\infty$. Hence $(y_i)_{i=1}^\infty\subseteq X_n$ is also equivalent to the u.v.b. of $c_0$. So there exists $c>0$ such that $\|y_i\|_n\ge c$ for all $i\ge 1$. Choose $A_i\in\mathcal{F}_n$ such that $\sum_{\gamma\in A_i}|e_\gamma^*(y_i)|\ge c$. Suppose there exists $M\in\mathbb{N}$ such that $A_i^{\#}=M$ for infinitely many $i$. For such $i$, we have $A_i=\mathop{\cup}\limits_{k=1}^M B_{i,k}$, where $B_{i,k}\in\mathcal{F}_{n-1}$. So for each such $i$, there exists $k(i)$ such that $\sum_{\gamma\in B_{i,k}}|e^*_\gamma(y_i)|\ge\frac cM$, and hence $\|y_i\|_{n-1}\ge\frac cM$. So $\{y_i \colon A^{\#}_i=M\}$ is equivalent to the u.v.b. of $c_0$ in $X_{n-1}$, which contradicts the inductive hypothesis.

On the other hand, if $\{i\colon A_i^{\#}=M\}$ is finite for each $M\in\mathbb{N}$, it follows that $A_i^{\#}\to\infty$ as $i\to\infty$. By the inclusion condition, since $S$ is finite, for every $M\in\mathbb{N}$ there exist $i_1<i_2<\cdots<i_M$ such that $A=\mathop{\cup}\limits_{k=1}^M A_{i_k}\in\mathcal{F}_{n+1}$. But then
$$
\Big\|\sum_{k=1}^M y_{i_k}\Big\|_{n+1}\ge \sum_{k=1}^M \sum_{\gamma\in A_{i_k}}|e^*_\gamma(y_{i_k})|\ge Mc,
$$
and hence 
$$
\Big\|\sum_{k=1}^M y_{i_k}\Big\|_X\ge \frac{Mc}{2^{n+1}}.
$$
Since $M$ is arbitrary, this contradicts the fact that $(y_i)_{i=1}^\infty$ is equivalent to the u.v.b. of $c_0$.
\end{proof}

\begin{theorem}\label{t42}
Suppose $S$ is finite. If $(\mathcal{F}_n)_{n=1}^\infty$ satisfy the inclusion condition then $X$ is reflexive.
\end{theorem}

\begin{proof}
Let $Z$ be any infinite-dimensional subspace of $X$. Then, since by Lemma \ref{l41} each $P_n$ is strictly singular, it follows that $Z$ contains a basic sequence that is equivalent to a block basis of $Y$. Hence $Z$ contains a reflexive subspace. In particular, $Z$ is not isomorphic to $c_0$ of $\ell_1$. Since $X$ has an unconditional basis, it follows from James's theorem  \cite{J1} that $X$ is reflexive.
\end{proof}

\begin{corollary}\label{cor43}
Let $S$ be finite. If $(\mathcal{F}_n)$ satisfy the inclusion condition, in particular if $(\mathcal{F}_n)$ are maximal, then $X$ admits an equivalent $1$-unconditional $2R$ norm.
\end{corollary} 
\begin{remark} It can be shown that the families used in \cite{AM} do \textit{not} satisfy the inclusion condition and yet the space $X$ is reflexive \cite{AM}.
Hence the inclusion condition is not a necessary condition for reflexivity.
\end{remark}
\begin{proposition}\label{prop44}
If $S$ is infinite then $X$ contains $c_0$.
\end{proposition}

\begin{proof}
For each $s\in S$, set $\gamma_s=(s,s,s,\dots)$. Then $\{1_{\gamma_s}\colon s\in S\}\subseteq X$ is equivalent to the u.v.b. of $c_0(S)$.
\end{proof}


\begin{thebibliography}{99}
\bibitem{AA} Dale E. Alspach and Spiros Argyros, \textit{Complexity of weakly null sequences}, Diss. Math. \textbf{321} (1992), 1--44.

\bibitem {AM} Spiros A. Argyros and Pavlos Motakis, \textit{$\alpha$-Large families and applications to Banach space theory}, Topology and its Applications \textbf{172} (2014), 47--67.

\bibitem{B} Albert Baernstein II, \textit{On reflexivity and summability}, Studia Math. \textbf{42} (1972), 91--94.


\bibitem{Beau} Bernard Beauzamy, \textit{Espaces d'interpolation r\'eels, topologie et g\'eom\'etrie}. Lecture Notes in Mathematics, vol. 666, Springer-
Verlag, Berlin-New York,1978.

\bibitem{BS} Y. Benyamini and T. Starbird, \textit{Embedding weakly compact sets into Hilbert space}, Israel J. Math. \textbf{23} (1970), 137--141.

\bibitem{DFJP} W. J. Davis, T. Figiel, W. B. Johnson, and A. Pe\l czy\'nski, \textit{Factoring Weakly Compact Operators}, J. Funct. Anal. {\bf 17} (1974), 311--327.

\bibitem{D} Mahlon M. Day, \textit{Strict convexity and smoothness},Trans. Amer. Math. Soc. {\bf 78} (1955), 516--528.
%
\bibitem{DGZ} Robert Deville, Gilles Godefroy, V\'aclav Zizler, \textit{Smoothness and renormings in Banach spaces},
Monographs and Surveys in Pure and Applied Mathematics, Vol. 64, Pitman, London, 1993.

\bibitem{DK} Stephen Dilworth and Denka Kutzarova, \textit{$2$-rotund norms for Baernstein spaces and their duals}, Pure and Applied Functional Analysis, {\bf 10} (2025), no. 5, 1229--1244.

\bibitem{DKM} Stephen Dilworth, Denka Kutzarova, and Pavlos Motakis,
 \textit{$2$-rotund norms for unconditional and symmetric sequence spaces},
 Banach J. Math. Anal., {\bf 18} (2024), no. 4, Paper No. 74, 17 pp.

\bibitem{FJ} T. Figiel and W. B. Johnson, \textit{A uniformly convex Banach space which contains no $\ell_p$}, Compositio Math. 29 (1974), 179--190.

\bibitem{G} Gilles Godefroy, \textit{Renormings of Banach spaces},  Handbook of the geometry of Banach spaces, Vol. I, 781--835, North Holland Publishing Co., Amsterdam, 2001.

\bibitem{HJ}  Petr H\'ajek and Michal Johanis, \textit{Characterization of reflexivity by equivalent  renorming}, J. Funct. Anal. \textbf{211} (2004), 163--172.

 \bibitem{J1}  R.C. James,\textit{ Bases and reflexivity of Banach spaces}, Ann. Math. \textbf{52} (1950), 518--527.

\bibitem{J} Robert C. James, \textit{Reflexivity and the sup of linear functionals}, Israel J. Math. \textbf{13} (1972), 289--300.

\bibitem{KT}  D.N. Kutzarova and S.L.Troyanski, \textit{Reflexive Banach spaces without equivalent norms that are uniformly convex or uniformly differentiable in every direction},  Studia Math. \textbf{72} (1982), 91--95.

\bibitem{LT1}   Joram Lindenstrauss and Lior Tzafriri, \textit{Classical Banach spaces. I}, Springer-Verlag, Berlin-New York, 1977.

\bibitem{LAT} J. L\'opez-Abad, S. Todorcecic, \textit{Positional graphs and conditional structure of weakly null sequences}, Adv. Math. {\bf 242} (2013), 163--186.

\bibitem{M} V.D. Milman, \textit{Geometric theory of Banach spaces. II: Geometry of the unit sphere}, Uspehi Mat. Nauk \textbf{26} (1971), 73--149.

\bibitem{OS} E. Odell and Th. Schlumprecht, \textit{Asymptotic properties of Banach spaces under renormings}, J. Amer. Math. Soc. \textbf{11} (1998), 175--188.

\bibitem{S} J. Schreier, \textit{Ein Gegenbeispiel zur Theorie der schwachen Konvergenz}, Stud. Math. {\bf 2} (1930), 58--62. 

\bibitem{T} B. S. Tsirelson, \textit{Not every Banach space contains an imbedding of $\ell_p$ or $c_0$}, Funct. Anal. Appl. {\bf 8} (1974), 138--141 (translated from Russian).
\end{thebibliography}
\end{document}